\documentclass[a4paper,12pt,reqno]{amsart}
\usepackage[T1]{fontenc}
\usepackage[utf8]{inputenc}
\usepackage{amssymb,dsfont,mathrsfs}
\usepackage[margin=1in]{geometry}
\usepackage{enumitem}
\usepackage[bookmarksdepth=2]{hyperref}

\newtheorem{theorem}{Theorem}
\newtheorem{lemma}[theorem]{Lemma}

\theoremstyle{definition}

\DeclareMathOperator{\im}{Im}

\DeclareMathOperator{\ind}{\mathds{1}}

\newcommand{\pr}{\mathds{P}}

\newcommand{\C}{\mathds{C}}
\newcommand{\R}{\mathds{R}}

\newcommand{\ph}{\varphi}

\newcommand{\expr}[1]{\left( #1 \right)}
\newcommand{\bb}{{\mathrm{b}}}
\newcommand{\oo}{{\mathrm{o}}}
\newcommand{\si}{{\mathrm{s}}}

\renewcommand{\ge}{\geqslant}

\usepackage{ifthen}
\newcommand{\formula}[2][nolabel]%
{%
 \ifthenelse{\equal{#1}{nolabel}}%
 {\begin{align*} #2 \end{align*}}%
 {%
  \ifthenelse{\equal{#1}{}}%
  {\begin{align} #2 \end{align}}%
  {\begin{align} \label{#1} #2 \end{align}}%
 }%
}

\begin{document}

\title[Random walks are determined\ldots]{Random walks are determined \\ by their trace on the positive half-line}
\author{Mateusz Kwaśnicki}
\thanks{Work supported by the Polish National Science Centre (NCN) grant no.\@ 2015/19/B/ST1/01457}
\address{Mateusz Kwaśnicki \\ Faculty of Pure and Applied Mathematics \\ Wrocław University of Science and Technology \\ ul. Wybrzeże Wyspiańskiego 27 \\ 50-370 Wrocław, Poland}
\email{mateusz.kwasnicki@pwr.edu.pl}

\begin{abstract}
We prove that the law of a random walk $X_n$ is determined by the one-dimensional distributions of $\max(X_n, 0)$ for $n = 1, 2, \ldots$\,, as conjectured recently by Loïc Chaumont and Ron Doney. Equivalently, the law of $X_n$ is determined by its upward space-time Wiener--Hopf factor. Our methods are complex-analytic.
\end{abstract}

\maketitle

%
%

\section{Introduction and main result}

In this note we give an affirmative answer to the question posed by Loïc Chaumont and Ron Doney in~\cite{bib:cd17}, inspired by Vincent Vigon's conjecture in~\cite{bib:v02}. The main result was previously stated without proof in a more general form in~\cite{bib:ou90}, and an erroneous proof was given in~\cite{bib:u94}.

A random walk $X_n$ is said to be \emph{non-degenerate} if $\pr(X_n > 0) \ne 0$. Similarly, a finite signed Borel measure $\mu$ on $\R$ is said to be \emph{non-degenerate} if the restriction of $\mu$ to $(0, \infty)$ is a non-zero measure.

\begin{theorem}
\label{thm:wh}
If $X_n$ and $Y_n$ are non-degenerate random walks such that $\max(X_n, 0)$ and $\max(Y_n, 0)$ are equal in distribution for all $n = 1, 2, \ldots$\,, then $X_n$ and $Y_n$ are equal in distribution for $n = 1, 2, \ldots$

More generally, if $\mu$ and $\nu$ are non-degenerate finite signed Borel measures and their $n$-fold convolutions $\mu^{*n}$ and $\nu^{*n}$ agree on $(0, \infty)$ for $n = 1, 2, \ldots$\,, then $\mu = \nu$.
\end{theorem}

Following~\cite{bib:cd17}, we remark that various reformulations of the above result are possible. A non-degenerate random walk $X_n$ is determined by any of the following objects:
\begin{itemize}
\item The law of the ascending ladder process $(T_k, S_k)$; here $S_k = X_{T_k}$ is the $k$-th running maximum of the random walk.
\item The upward space-time Wiener--Hopf factor $\Phi_+(q, \xi)$, that is, the characteristic function of $(T_1, S_1)$.
\item The distributions of the running maxima $\max(0, X_1, X_2, \ldots, X_n)$ for $n = 1, 2, \ldots$
\end{itemize}
Theorem~\ref{thm:wh} clearly implies that a~non-degenerate Lévy process $X_t$ is determined by any of the following objects:
\begin{itemize}
\item The distributions of $\max(X_t, 0)$ for all $t > 0$ (or even for $t = 1, 2, \ldots$).
\item The law of the ascending ladder process $(T_t, S_t)$.
\item The upward space-time Wiener--Hopf factor $\kappa_+(q, \xi)$, that is, the characteristic exponent of $(T_t, S_t)$.
\item The distributions of the running suprema $\sup\{X_s : s \in [0, t]\}$ for all $t > 0$.
\end{itemize}
For further discussion, we again refer to~\cite{bib:cd17}, where Theorem~\ref{thm:wh} was proved under various relatively mild additional conditions. For related research, see~\cite{bib:cd17,bib:lms76,bib:o85,bib:ou90,bib:u90,bib:u94} and the references therein.

Theorem~\ref{thm:wh} was given without proof in~\cite{bib:ou90} in a more general form: Theorem~4 therein claims that $\mu = \nu$ if $\mu$ and $\nu$ are non-degenerate finite Borel measures on $\R$ and the restrictions of $\mu^{*n_k}$ and $\nu^{*n_k}$ to $(0, \infty)$ are equal for $k = 1, 2, \ldots$\,, where $n_1 = 1$ and $n_2 - 1, n_3 - 1, \ldots$ are distinct and have no common divisor other than $1$. Noteworthy, this result is stated for measures on the Euclidean space of arbitrary dimension, and their restrictions to the half-space. A proof is given in~\cite{bib:u94} under the additional condition $n_2 = 2$, and only in dimension one. However, the argument in~\cite{bib:u94} contains a gap, that we describe at the end of this article.

%
%

\section{Proof}

All measures considered below are finite, signed Borel measures. For a measure $\mu$ on $\R$, we denote the restrictions of $\mu$ to $(0, \infty)$ and $(-\infty, 0]$ by $\mu_+ = \ind_{(0, \infty)} \mu$ and $\mu_- = \ind_{(-\infty, 0]} \mu$. This should not be confused with the Hahn decomposition of $\mu$ into the positive and negative part. By $\mu^{*n}$ we denote the $n$-fold convolution of $\mu$, and we define $\mu^{*0}$ to be the Dirac measure $\delta_0$. For brevity, we write $\mu_\pm^{*n} = (\mu_\pm)^{*n}$, as opposed to $(\mu^{*n})_\pm$. We denote the characteristic function of a measure $\mu$ by $\hat\mu$:
\formula{
 \hat\mu(z) & = \int_\R e^{i z x} \mu(dx)
}
for $z \in \R$, and also for those $z \in \C$ for which the integral converges. We recall that $\hat\mu_+$ is a bounded holomorphic function in the upper complex half-plane $\C_+ = \{z \in \C : \im z > 0\}$, continuous on the boundary. Similarly, $\hat\mu_-$ is a bounded holomorphic function on the lower complex half-plane $\C_- = \{z \in \C : \im z < 0\}$.

\begin{lemma}
\label{lem:mupm}
Suppose that $\mu, \nu$ are measures on $\R$ satisfying
\formula{
 (\mu^{*n})_+ & = (\nu^{*n})_+ && \text{for $n = 1, 2, \ldots, N$.}
}
Then $\mu_+ = \nu_+$ and
\formula[eq:mupm]{
 (\mu_+^{*n} * \mu_-^{*k})_+ & = (\nu_+^{*n} * \nu_-^{*k})_+ && \text{for $n = 1, \ldots, N - 1$ and $k = 1, 2, \ldots$}
}
\end{lemma}

\begin{proof}
We proceed by induction with respect to $N$. For $N = 1$ the result is trivial: we have $\mu_+ = (\mu^{*1})_+ = (\nu^{*1})_+ = \nu_+$. Suppose that the assertion of the lemma holds true for some $N$, and suppose that $(\mu^{*n})_+ = (\nu^{*n})_+$ for $n = 1, 2, \ldots, N, N + 1$. By the induction hypothesis, formula~\eqref{eq:mupm} holds for $n = 0, 1, \ldots, N - 1$ and $k = 1, 2, \ldots$\,, and we have $\mu_+ = \nu_+$. Therefore, we only need to prove~\eqref{eq:mupm} for $n = N$ and $k = 1, 2, \ldots$

By the binomial theorem,
\formula{
 0 & = (\mu^{*N + 1} - \nu^{*N + 1})_+ = ((\mu_+ + \mu_-)^{*N + 1} - (\nu_+ + \nu_-)^{*N + 1})_+ \\
 & = \sum_{j = 0}^{N + 1} \binom{N + 1}{j} (\mu_+^{*j} * \mu_-^{*N + 1 - j} - \nu_+^{*j} * \nu_-^{*N + 1 - j})_+ .
}
We already know that $\mu_+^{*N + 1} = \nu_+^{*N + 1}$ and $(\mu_+^{*j} * \mu_-^{*N + 1 - j})_+ = (\nu_+^{*j} * \nu_-^{*N + 1 - j})_+$ for $j = 1, 2, \ldots, N - 1$. Furthermore, $(\mu_-^{*N + 1})_+ = 0 = (\nu_-^{*N + 1})_+$. It follows that all terms corresponding to $j \ne N$ in the above sum are zero. Thus,
\formula{
 0 & = \binom{N + 1}{N} (\mu_+^N * \mu_- - \nu_+^N * \nu_-)_+ ,
}
which proves~\eqref{eq:mupm} for $n = N$ and $k = 1$. The proof for $n = N$ and $k > 1$ proceeds again by induction: if~\eqref{eq:mupm} holds for $n = N$ and $k = 1, 2, \ldots, K$, then
\formula{
 (\mu_+^N * \mu_-^{*K + 1})_+ & = (\mu_+^N * \mu_-^{*K} * \mu_-)_+ = ((\mu_+^N * \mu_-^{*K})_+ * \mu_-)_+ \\
 & = ((\nu_+^N * \nu_-^{*K})_+ * \mu_-)_+ = (\nu_+^N * \nu_-^{*K} * \mu_-)_+ \\
 & = (\mu_+^N * \mu_- * \nu_-^{*K})_+ = ((\mu_+^N * \mu_-)_+ * \nu_-^{*K})_+ \\
 & = ((\nu_+^N * \nu_-)_+ * \nu_-^{*K})_+ = (\nu_+^N * \nu_- * \nu_-^{*K})_+ = (\nu_+^N * \nu_-^{*K + 1})_+ ;
}
we used the identity $(\pi * \sigma_-)_+ = (\pi_+ * \sigma_-)_+$ in second, fourth, sixth and eighth equalities, \eqref{eq:mupm} for $n = N$ and $k = K$ in the third one, \eqref{eq:mupm} for $n = N$ and $k = 1$ in the seventh one, and $\mu_+ = \nu_+$ in the fifth one. We conclude that~\eqref{eq:mupm} holds for $n = N$ and all $k = 1, 2, \ldots$\,, and the proof is complete.
\end{proof}

A holomorphic function $f$ on $\C_-$ is said to be of \emph{bounded type} (or belong to the \emph{Nevanlinna class}) if $\log |f(x)|$ has a harmonic majorant on $\C_-$. Equivalently, $f$ is of bounded type if it is a ratio of two bounded holomorphic functions on $\C_-$. We recall the following fundamental factorisation theorem for holomorphic functions on $\C_-$ which are bounded or of bounded type, and we refer to~\cite{bib:g07} for further details.

\begin{theorem}[{Theorem~II.5.5 and Corollary~II.5.7 in~\cite{bib:g07}}]
\label{thm:fact}
Let $f$ be a holomorphic function of bounded type on the lower complex half-plane, and suppose that $f$ is not identically zero. Let $\alpha_0$ be the multiplicity of the zero of $f$ at $z = -i$ (possibly $n_0 = 0$), and let $z_1, z_2, \ldots$ be the (finite or infinite) sequence of all zeros of $f$ in the lower complex half-plane, with corresponding multiplicities $\alpha_1, \alpha_2, \ldots$\, Then $f$ admits a factorisation
\formula[eq:fact]{
 f(z) & = f_\bb(z) f_\oo(z) f_\si(z)
}
(unique, up to multiplication of $f_\oo$ and $f_\si$ by a constant of modulus~$1$), with the folllowing factors. The function $f_\bb$ is a \emph{Blaschke product}, determined uniquely by the zeros of $f$:
\formula[eq:bb]{
 f_\bb(z) & = \expr{\frac{z + i}{z - i}}^{\!\alpha_0} \prod_j \expr{\frac{|1 + z_j^2|}{1 + z_j^2} \, \frac{z - z_j}{z - \bar{z}_j}}^{\!\alpha_j} .
}
The function $f_\oo$ is an \emph{outer function}, a holomorphic function determined uniquely up to multiplication by a constant of modulus~$1$ by the formula:
\formula[eq:oo]{
 |f_\oo(z)| & = \exp \! \expr{\frac{1}{\pi} \int_{-\infty}^\infty \frac{-\im z}{|z - x|^2} \, \log |f(x)| \, dx} .
}
Finally, the function $f_\si$ is a \emph{singular inner function}, a holomorphic function determined uniquely up to multiplication by a constant of modulus~$1$ by the expression:
\formula[eq:si]{
 |f_\si(z)| & = \exp \! \expr{a \im z - \frac{1}{\pi} \int_\R \frac{-\im z}{|z - x|^2} \, \lambda(dx)} ,
}
where $a \in \R$ is a constant and $\lambda$ is a signed measure, singular with respect to the Lebesgue measure.

Furthermore, for almost all $x \in \R$ with respect to both the Lebesgue measure and the measure $\lambda$, the limit $f(x)$ of $f(x + i y)$ as $y \to 0^-$ exists. This boundary limit $f(x)$ is non-zero almost everywhere with respect to the Lebesgue measure and zero almost everywhere with respect to $\lambda$. The symbol $f(x)$ used in the definition of the outer function $f_\oo$ refers precisely to this boundary limit. Additionally, we have $\sum_j \alpha_j |\im z_j| (1 + |z_j|^2)^{-1} < \infty$, $\int_{-\infty}^\infty (1 + x^2)^{-1} |\log |f(x)|| dx < \infty$ and $\int_\R (1 + x^2)^{-1} |\lambda|(dx) < \infty$, and any parameters $\alpha_j$, $z_j$, $a$, $\lambda$ and boundary values $|f(x)|$, $x \in \R$, which satisfy these conditions, correspond to some function $f$ of bounded type.

Finally, $f$ is a bounded holomorphic function in the lower complex half-plane if and only if $a \ge 0$, $\lambda$ is a non-negative measure and the boundary values $|f(x)|$ are bounded for $x \in \R$.
\end{theorem}

\begin{lemma}
\label{lem:ext}
Suppose that $\mu$ is a measure on $\R$ such that $\mu_-$ is a non-zero measure and $(\mu_+ * \mu_-)_+ = 0$. Then $\hat\mu_+$ has a holomorphic extension $\ph$ to the the connected open set
\formula{
 D & = \C \setminus \{z \in \C_- \cup \R : \hat\mu_-(z) = 0\} ,
}
and $\ph$ is a meromorphic function on $\C \setminus \{z \in \R : \hat\mu_-(z) = 0\}$. Furthermore, $\ph \hat\mu_-$ extends to a function which is holomorphic on $\C_-$ and continuous on $\C_- \cup \R$, namely, the characteristic function of $\mu_+ * \mu_-$.
\end{lemma}

\begin{proof}
Denote $\nu = \mu_+ * \mu_-$; by the assumption, $\nu = \nu_-$. Let $f = \hat\mu_+$, $g = \hat\mu_-$ and $h = \hat \nu = \hat \nu_-$. Clearly, $h(z) = f(z) g(z)$ for $z \in \R$. Let
\formula{
 A & = \{ z \in \R : g(z) = 0 \} , & B = \{ z \in \C_- : g(z) = 0 \} ,
}
so that $D = \C \setminus (A \cup B)$.

We note basic properties of $A$ and $B$. By continuity of $g$, $A$ and $A \cup B$ are closed sets, and $D$ is an open set. Since $g$ is holomorphic on $\C_-$ (and not identically zero), $B$ is a countable (possibly finite) set with no accumulation points on $\C_-$. By Theorem~\ref{thm:fact}, $A$ has zero Lebesgue measure (as a subset of $\R$). In particular, $D$ is connected.

We define a function $\ph$ on $D$ by the formula
\formula{
 \ph(z) & = \begin{cases} f(z) & \text{if $z \in \C_+ \cup (\R \setminus A)$,} \\[0.5em] \dfrac{h(z)}{g(z)} & \text{if $z \in \C_- \setminus B$.} \end{cases}
}
By definition, $\ph$ is holomorphic both on $\C_+$ and on $\C_- \setminus B$, as well as meromorphic on $\C_-$. Furthermore, $\ph$ is continuous at each point $z \in \R \setminus A$, because both $f$ (defined on $\C_+ \cup \R$) and $h / g$ (defined on $(\C_- \setminus B) \cup (\R \setminus A)$) are continuous at $z$ and $f(z) = h(z) / g(z)$. By a standard application of Morera's theorem, $\ph$ is holomorphic in $D$. It remains to note that $\ph(z) g(z) = h(z)$ for $z \in \C_- \setminus B$.
\end{proof}

\begin{lemma}
\label{lem:wh}
If $\mu$ is a measure on $\R$ such that $(\mu_+^{*n} * \mu_-)_+ = 0$ for all $n = 1, 2, \ldots$\,, then either $\mu_+$ or $\mu_-$ is a zero measure.
\end{lemma}

\begin{proof}
Let $\mu$ be such a measure, and suppose that both $\mu_+$ and $\mu_-$ are non-zero measures. Let $\ph, f, g, h, A, B, D$ be as in the proof of Lemma~\ref{lem:ext}. Clearly, $\ph^n$ is the holomorphic extension of $f^n$, the characteristic function of $\mu_+^{*n}$. An application of Lemma~\ref{lem:ext} to the measure $\mu_+^{*n} + \mu_-$ implies that for all $n = 1, 2, \ldots$\,, the function $\ph^n g$ extends from $\C_- \setminus B$ to a function $h_n$ which is bounded and holomorphic on $\C_-$ and continuous on $\C_- \cup \R$, namely, $h_n$ is the characteristic function of $\mu_+^{*n} * \mu_-$.

Consider the factorisations $g = g_\bb g_\oo g_\si$ and $h_n = h_{n,\bb} h_{n,\oo} h_{n,\si}$ given in Theorem~\ref{thm:fact}, and let $\lambda_g$, $a_g$ and $\lambda_{h,n}$, $a_{h,n}$ denote the corresponding non-negative measures $\lambda$ and constants $a$ for $g$ and $h_n$, respectively. Note that Theorem~\ref{thm:fact} applies both to $g$ and to $h_n = \ph^n g$, as these functions are not identically zero: $f$ and $g$ are characteristic functions of non-zero measures $\mu_+$ and $\mu_-$, while $h_n$ is the product of $g$ and the holomorphic extension of $f^n$.

Recall that $\ph^n = h_n / g$ on $\C_- \setminus B$. It follows that if $\ph_{n,\bb} = h_{n,\bb} / g_\bb$, $\ph_{n,\oo} = h_{n,\oo} / g_\oo$ and $\ph_{n,\si} = h_{n,\si} / g$, then
\formula{
 \ph^n & = \ph_{n,\bb} \ph_{n,\oo} \ph_{n,\si}
}
on $\C_- \setminus B$. Let us examine the above factors in more detail.

By definition, $\ph_{n,\oo}$ and $\ph_{n,\si}$ have no zeros in $\C_-$. This means that if $z_0 \in \C_-$ is a pole of $\ph$ of order $\alpha_0$, then $z_0$ is a pole of $\ph_{n,\bb} = h_{n,\bb} / g_\bb$ of order $n \alpha_0$, and therefore $g_\bb$ has a zero at $z_0$ of multiplicity at least $n \alpha_0$ for all $n = 1, 2, \ldots$\, Since all zeroes of $g_\bb$ have finite multiplicity, $\ph$ has no poles in $\C_-$. In particular, $\ph$ extends to a holomorphic function on $\C \setminus A$, which will be denoted again by $\ph$, and $\ph_{n,\bb} = h_{n,\bb} / g_\bb$ has no poles in $\C_-$. Therefore, the zeros of $h_{n,\bb}$ must cancel the zeros of $g_\bb$, and $\ph_{n,\bb}$ is a Blaschke product.

Since $h_n(x) / g(x) = (f(x))^n$ for $x \in \R \setminus A$ and $A$ has Lebesgue measure zero, we have
\formula{
 |\ph_{n,\oo}(z)| & = \exp \! \expr{\frac{1}{\pi} \int_{-\infty}^\infty \frac{-\im z}{|z - x|^2} \, (\log |h_n(x)| - \log |g(x)|) \, dx} \\
 & = \exp \! \expr{\frac{1}{\pi} \int_{-\infty}^\infty \frac{-\im z}{|z - x|^2} \, \log |f(x)|^n \, dx} .
}
In particular, $\ph_{n,\oo}$ is a bounded outer function, namely, the outer function in the factorisation of the bounded holomorphic function $(\overline{f(\bar{z})})^n$ on the lower complex half-plane.

Finally $\ph_{n,\si}$ is the ratio of two singular inner functions, and hence a singular inner function. If we denote $a_{\ph,n} = a_{h,n} - a_g$ and $\lambda_{\ph,n} = \lambda_{h,n} - \lambda_g$, then
\formula{
 |\ph_{n,\si}(z)| & = \exp \! \expr{-a_{\ph,n} \im z - \frac{1}{\pi} \int_\R \frac{-\im z}{|z - x|^2} \, \lambda_{\ph,n}(dx)} .
}

The above properties imply that $\ph^n$ is of bounded type, and therefore the factors $\ph_{n,\bb}$, $\ph_{n,\oo}$, $\ph_{n,\si}$, the signed measure $\lambda_{\ph,n}$ and the constant $a_{\ph,n} \in \R$ are uniquely determined (up to multiplication by a constant of modulus $1$ in case of $\ph_{n,\oo}$ and $\ph_{n,\ss}$).

By comparing the factorisations of $\ph$ and $\ph^n$, we find that $\ph_{n,\si} = c_n (\ph_{1,\si})^n$ for some constant $c_n$ with modulus $1$. It follows that $a_{\ph,n} = n a_{\ph,1}$ and $\lambda_{\ph,n} = n \lambda_{\ph,1}$. This, however, implies that $a_{\ph,1} = \tfrac{1}{n} a_{\ph,n} \ge - \tfrac{1}{n} a_g$ for all $n = 1, 2, \ldots$\,, and so $a_{\ph,1} \ge 0$. Similarly, the negative part of $\lambda_{\ph,1} = \tfrac{1}{n} \lambda_{\ph,n}$ is dominated by $\tfrac{1}{n} \lambda_g$ for any $n = 1, 2, \ldots$\, This is not possible if the negative part of $\lambda_{\ph,1}$ is non-zero, and therefore $\lambda_{\ph,1}$ is a non-negative measure. We conclude that $\ph = \ph_{1,\bb} \ph_{1,\oo} \ph_{1,\si}$ is a bounded holomorphic function on $\C_-$.

Since $\ph = f$ on $\C_+$ and $f$ is a bounded holomorphic function on $\C_+$, we have proved that $\ph$ is a bounded holomorphic function on $\C \setminus A$. However, $A$ has zero Lebesgue measure (as a subset of $\R$). By Painlevé's theorem (see Theorem~2.7 in~\cite{bib:y15}), $\ph$ extends to a bounded holomorphic function on $\C$. This, in turn, implies that $\ph$ is constant, and so $\hat\mu_+$ is constant, contradicting the assumption that $\mu_+$ is a non-zero measure on $(0, \infty)$.
\end{proof}

\begin{proof}[Proof of Theorem~\ref{thm:wh}]
Suppose that $(\mu^{*n})_+ = (\nu^{*n})_+$ for $n = 1, 2, \ldots$ for some measures $\mu$ and $\nu$ such that $\mu_+$ and $\nu_+$ are non-zero measures. By Lemma~\ref{lem:mupm}, $\mu_+ = \nu_+$ and $(\mu_+^{*n} * \mu_-)_+ = (\nu_+^{*n} * \nu_-)_+$ for $n = 1, 2, \ldots$\, Let $\eta = \mu_+ + \mu_- - \nu_-$, so that $\eta_+ = \mu_+ = \mu_-$ and $\eta_- = \mu_- - \nu_-$. Then $(\eta_+^{*n} * \eta_-)_+ = 0$ for $n = 1, 2, \ldots$\,, and therefore, by Lemma~\ref{lem:wh}, either $\eta_+$ or $\eta_-$ is a zero measure. Since $\eta_+ = \mu_+$ is a non-zero measure, we must have $\eta_- = 0$, that is, $\mu_- = \nu_-$.
\end{proof}

%
%

\section{An error in~\cite{bib:u94}}

In~\cite{bib:u94} an analogue of Theorem~\ref{thm:wh} is given, with equality of $\mu^{*n}$ and $\nu^{*n}$ on $(-\infty, 0)$ rather than on $(0, \infty)$. In page~3001, line~16 of~\cite{bib:u94}, it is claimed that the measures $\mu$ and $\nu$ satisfy condition~(B) of Theorem~A in~\cite{bib:u94}, as a consequence of the results of Section~11.2 in~\cite{bib:lo77}. This reasoning would have been correct if the holomorphic extensions of $\hat{\mu}$ and $\hat{\nu}$ to the upper complex half-plane had been known to be continuous on the boundary. However, this is not verified in~\cite{bib:u94}.

More precisely, it is observed in~\cite{bib:u94} that $\hat{\mu} = (\hat{\chi}_2 - (\hat{\chi}_1)^2) / (2 \hat{\chi}_1)$ almost everywhere on~$\R$, where $\chi_1 = \mu - \nu$ and $\chi_2 = \mu^{*2} - \nu^{*2}$ are measures concentrated on $(0, \infty)$. Since $\hat{\chi}_1$ and $\hat{\chi}_2$ extend to holomorphic functions on $\C_+$, $\hat{\mu}$ extends to a meromorphic function on $\C_+$. Equality of $\mu^{*n}$ and $\nu^{*n}$ on $(-\infty, 0)$ for $n \ge 3$ is used only to show that the extension of $\hat{\mu}$ has no poles in $\C_+$. However, the extension of $\hat{\mu}$ can have singularities near $\R$ and thus fail to satisfy condition~(B) of Theorem~A in~\cite{bib:u94}.

To be specific, observe that $\hat{\mu}(z) = z^2 (z + i)^{-4} \exp(i/z)$ is the characteristic function of a measure $\mu$ on $\R$. Namely, $\mu$ is the convolution of $\tfrac{1}{6} x^3 e^{-x} \ind_{(0, \infty)}(x) dx$ and $\tfrac{1}{6} \, {_0F_1}(4; x) \ind_{(-\infty, 0)}(x) dx - \tfrac{1}{2} \delta_0(dx) - \delta_0'(dx) - \delta_0''(dx)$ (in the sense of distributions; ${_0F_1}$ is the hypergeometric function; we omit the details). Clearly, $\hat{\mu}$ extends holomorphically to the upper complex half-plane, but this extension is not continuous on the boundary, and thus $\mu$ does not satisfy condition~(B) of Theorem~A in~\cite{bib:u94}. Furthermore, $\hat{\mu}(z)$ is the ratio of two characteristic functions of finite measures supported in $[0, \infty)$: $z^4 / (z + i)^8$ and $z^2 (z + i)^{-4} \exp(-i/z)$.

The author of the present article was not able to correct the error in~\cite{bib:u94}. The proof given above uses a related, but essentially different idea.

%
%

\medskip

\subsection*{Acknowledgments}
I thank Loïc Chaumont for numerous discussions on the subject of the article and encouragement. I thank Jan Rosiński for letting me know about reference~\cite{bib:ou90}. I thank Alexander Ulanovski\u{\i} for discussions on references~\cite{bib:ou90,bib:u94}. The main part of this article was written during the \emph{39th Conference on Stochastic Processes and their Applications} in Moscow.

%
%

%
%

\end{document}